\documentclass[12pt,a4paper]{amsart}
\usepackage{amsfonts}
\usepackage[active]{srcltx}

\usepackage{enumerate}
\usepackage[notref,notcite]{showkeys}
\usepackage{amsmath,amssymb,xspace,amsthm}

\newtheorem{theorem}{Theorem}[section]

\newtheorem{lemma}[theorem]{Lemma}

\numberwithin{equation}{section} \theoremstyle{definition}

\def\span{\operatorname{span}}

\newcommand{\ZZ}{{\mathbb Z}}

\newcommand{\C}{\ensuremath{\mathbb C}\xspace}

\renewcommand{\a}{\ensuremath{\alpha}}

\newcommand{\h}{\ensuremath{\mathfrak{H}}}

\newcommand{\Z}{\ensuremath{\mathbb{Z}}\xspace}

\renewcommand{\phi}{\varphi}

\renewcommand{\leq}{\leqslant}

\def\mu{\mathfrak{u}}

\def\sl{\mathfrak{sl}}
\def\gl{\mathfrak{gl}}

\def\span{\text{span}}

\begin{document}
\title[Irreducible Witt modules]{Irreducible weight  modules over Witt algebras with
infinite dimensional weight spaces}
\author{Genqiang Liu and Kaiming zhao}
\date{}\maketitle
\begin{abstract} Let $d>1$ be an integer.
In 1986, Shen defined a class of weight modules $F^\alpha_b(V)$ over
the Witt algebra $\mathcal{W}_d$ for  $\a\in\C^d$, $b\in\C$, and an
irreducible module $ V$ over the special linear Lie algebra $\sl_d$.
In 1996, Eswara Rao determined the necessary and sufficient
conditions for these modules to be irreducible when $V$ is finite
dimensional. In this note, we will determine the necessary and
sufficient conditions for   all these modules $F^\alpha_b(V)$  to be
irreducible where $V$ is not necessarily finite dimensional.
Therefore we obtain a lot of irreducible $\mathcal{W}_d$-modules
with infinite dimensional weight spaces.
\end{abstract}

\vskip 5pt
\noindent
{\em 2010  Math. Subj. Class.:}
17B10, 17B20, 17B65, 17B66, 17B68

\vskip 10pt

\section{Introduction}

We denote by $\mathbb{Z}$, $\mathbb{Z}_+$, $\mathbb{N}$ and
$\mathbb{C}$ the sets of  all integers, nonnegative integers,
positive integers and complex numbers, respectively. Let $d>1$ be an integer.

Representation theory for  infinite-dimensional Lie algebras has
been attracting extensive attentions of many mathematicians and
physicists. These Lie algebras include the Witt algebras
$\mathcal{W}_d$ which is the derivation algebra of the Laurent
polynomial algebra $A_d=\C[x_1^{\pm1},x_2^{\pm1},\cdots,
x_d^{\pm1}]$. The algebra $\mathcal{W}_d$ is a natural higher rank
generalization of the Virasoro algebra, which has many applications
to different branches of mathematics and physics (see \cite{M,
L1,L2,L3,L4,L5}) and at the same time a much more complicated
representation theory.

The weight representation theory of Witt algebras was recently
studied by many experts; see \cite{B, E1, E2, BMZ, GLZ,L3, L4, L5,
MZ2,Z}. In 1986,  Shen defined a class of modules $F^\alpha_b(V)$
over the Witt algebra $\mathcal{W}_d$ for  $\a\in\C^d$, $b\in\C$,
and an irreducible module $ V$ over the special linear Lie algebra
$\sl_d$, see \cite{Sh}, which were also given by Larsson in 1992,
see \cite{L3}. In 1996, Eswara Rao determined the necessary and
sufficient conditions for   these modules to be irreducible when $V$
is finite dimensional, see \cite{E1}. A simplified proof was given
in \cite{GZ}. It is natural to study the irreducibility of  these
modules $F^\alpha_b(V)$ when $V$ is infinite dimensional. This is
what the present paper will do. In this manner we obtain a lot of
irreducible modules over the Witt algebra $\mathcal{W}_d$  with
infinite dimensional weight spaces.

We have to mention the remarkable work by Billig and Futorny
\cite{BF} in which they proved very recently that irreducible
modules for $\mathcal{W}_d$ with finite-dimensional weight spaces
fall in two classes: (1) modules of the highest weight type and (2)
irreducible modules from  $F^\alpha_b(V)$. So the next task in
representation theory of Witt algebras  is to study irreducible
modules over the Witt algebra $\mathcal{W}_d$  with infinite
dimensional weight spaces, and irreducible non-weight modules. This
is the second reason of this paper.

For irreducible weight modules (not necessary with finite
dimensional weight spaces) over $\mathcal{W}_d$, all weight sets
were explicitly determined in \cite{MZ2}, while some non-weight
irreducible $\mathcal{W}_d$-modules were constructed in \cite{TZ}.

In this paper we actually prove that,  for any   infinite
dimensional irreducible $\sl_d$-module  $V$, any $\alpha\in\C^d$ and
$b\in\C$, the $\mathcal{W}_d$-module $F^\alpha_b(V)$ is always
irreducible, see Theorem 2.4. We also determine the necessary and
sufficient conditions for two $\mathcal{W}_d$-module $F^\alpha_b(V)$
to be isomorphic, see Theorem 2.6.  The main trick we employed in
our proof is a characterization of finite dimensional irreducible
modules over $\sl_d$, see Lemma 2.3, which makes the proof  very
short, and very different from that for finite dimensional
irreducible $\sl_d$-module  $V$, see \cite{E1, GZ}.

Note that $F^\alpha_b(V)$ are polynomial modules in the sense of
\cite{BB,BZ}.

\section{$\mathcal{W}_d$-modules from $\sl_d$-modules}

For
 a positive integer $d>1$, $\ZZ^{d}$ denotes the direct sum of $d$ copies of $\ZZ$.
 For
any $a=(a_1,\cdots, a_d)^T \in \Z_+^d$ and $n=(n_1,\cdots,n_d)^T
\in\C^d$, we denote $n^{a}=n_1^{a_1}n_2^{a_2}\cdots n_d^{a_d}$,
where $T$ means taking transpose of the matrix.  Let $\gl_d$ be the
Lie algebra of all $d \times d$ complex matrices, $\sl_d$ the
subalgebra of $\gl_d$ consisting of all traceless matrices. For $1
\leq i, j \leq d$ we use $E_{ij}$ to denote the matrix units.
\subsection{ Witt algebras $\mathcal{W}_d$}
We  denote by  $ \mathcal{W}_d$ the derivation Lie algebra of the
Laurent polynomial algebra $A_d=\C[x_1^{\pm1},x_2^{\pm1}, . . .
,x_d^{\pm1}]$. For $i\in\{1,2,\dots,d\}$, set $\partial
_i=x_i\frac{\partial}{\partial x_i}$; and for any
$a=(a_1,a_2,\cdots, a_d)^T\in\mathbb{Z}^d$ (considered as column
vectors), set $x^a=x_1^{a_1}x_2^{a_2}\cdots x_d^{a_d}$.

We fix the vector space $\mathbb{C}^d$ of $d\times 1$ matrices.
Denote the standard basis by $\{e_1,e_2,...,e_d\}$. Let
$(\,\cdot\,|\, \cdot\, )$ be the standard symmetric bilinear form
such that $(u | v)=u^Tv\in\mathbb{C}$.  For $u \in \mathbb {C}^d$
and $r\in \mathbb{Z}^d$, we denote
$D(u,r)=x^r\sum_{i=1}^du_i\partial_i$. Then we have
$$[D(u,r),D(v,s)]=D(w,r+s),\,u,v\in \mathbb {C}^d, r,s\in \mathbb {Z}^d,$$
where $w=(u | s)v-(v | r)u$. Note that for any $u,v,z,y\in
\mathbb{C}^d$, both $uv^T$ and $xy^T$ are $d\times d$ matrices, and
 \begin{equation*}(uv^T)(zy^T)=(v|z)uy^T.\end{equation*}
 We know that $\h=\span\{\partial_1, \partial_2, ... , \partial_d\}$ is the Cartan
 sunalgebra of $\mathcal{W}_d$. A weight $\mathcal{W}_d$-module is a
 $\mathcal{W}_d$-module $V$ so that $V=\oplus_{\lambda\in\h^*}V_{\lambda}$
 where $$V_{\lambda}=\{v\in V : hv=\lambda(h)v \,\,\forall\,\, h\in\h\}.$$

\subsection{ Defining $\mathcal{W}_d$-modules} Let us first recall Shen's Witt modules from \cite{Sh}.
For  any  $\alpha\in \mathbb{C}^d, b\in\C$ and $\gl_d$-module $V$ on
which the  identity matrix acts as the scalar $b$, let
$F^\alpha_b(V)=V\otimes A_d$. For simplicity we write $v(n) = v
\otimes x^n$ for any $v\in V, n\in\Z^d$.
 Then $F^\alpha_b(V)$  becomes a
$\mathcal{W}_d$-module if we define the following actions
 \begin{equation}D(u,r)v(n)=\Big((u \mid n+\alpha)v+
 (ru^T)v\Big)(n+r),\end{equation}
where $u\in\mathbb{C}^d$, $v\in V$,  $ n, r\in\mathbb{Z}^d$. It is
easy to see that the module $F^\alpha_b(V)$ obtained from any $V$ is
always a weight module over $\mathcal{W}_d$.

When $V$ is finite dimensional, the following Theorem was given
 by Eswara Rao, see \cite{E2} or \cite{GZ}.

\begin{theorem}\label{t} Let $\alpha\in \mathbb{C}^d, b\in\C$, and $V$ be
 a finite  dimensional irreducible module over $\gl_d$ on  which the identity matrix acts as the scalar
$b$. Then $F^\alpha_b(V)$ is a reducible $\mathcal{W}_d$-module if
and only if the highest weight of $V$ is the fundament weight
$\omega_k$ of $\gl_d$ and $b=k$, where $k\in\Z$ with $1\leq k\leq
d-1$.
\end{theorem}

\subsection{Description of $F^\alpha_b(V)$}
The following useful lemma from linear algebra is obvious.

\begin{lemma}\label{l}   For $n=(n_1,n_2,...,n_d), u=(u_1,u_2,...,u_d)\in\C^d$,
let $$g(n,u)=\sum_{a,b \in\Z_+^d}c_{a,b}n^au^b$$ be a polynomial in
$2d$ variables $n_i, u_j$ with coefficients
$c_{a,b}\in\emph{End}(V)$. Let $V^*$ be a subspace of $ V$. For
$v\in V$,
   if $g(n,u)v\in V^* $ for all $n\in\Z^d,u\in\C^d$, then
  $c_{a,b}v\in V^*$ for any $a,b\in\Z_+^d$.
\end{lemma}

Now we need the following characterization on finite dimensional
irreducible $\sl_d$-modules.

\begin{lemma}\label{l2} Let $V$ be an irreducible  $\sl_d$-module(not necessarily weight
module).
\begin{enumerate}
\item For any $i,j : 1\leq i\neq j\leq d$, $E_{ij}$ acts injectively  or locally nilpotently on
$V$.

\item  The module $V$ is finite
dimensional iff $E_{ij}$ acts locally nilpotently on $V$ for any
$i,j : 1\leq i\neq j\leq d$.
\end{enumerate}
\end{lemma}

\begin{proof} (1). If $E_{ij}$ is not injective on $V$, then $E_{ij}v=0$ for some nonzero vector $v$
in $V$. It is known that $\text{ad}E_{ij}$ is locally nilpotent on
  the universal enveloping algebra $U(\sl_d)$ of $\sl_d$. This
means that,   for any  $X\in U(\sl_d)$ we have  $(\text{ad}E_{ij})^kX=0$
for sufficiently large $k$. Then from the identity
$$ E_{ij}^kXv=((\text{ad}E_{ij})^kX)v=0,$$
for sufficiently large $k$.  Since $V$ is irreducible, $V=U(\sl_d)v$. Thus
$E_{ij}$ acts locally nilpotently on $V$.

(2). If $V$ is finite dimensional, clearly $E_{ij}$ acts locally
nilpotently on $V$ for any $i,j : 1\leq i\neq j\leq d$.

Now suppose that   $E_{ij}$ acts locally nilpotently on $V$ for any
$i,j : 1\leq i\neq j\leq d$. By the condition that $E_{ij}$ acts
locally nilpotently on $V$ for any $i,j\in\Z$ with $1\leq i < j\leq
d$ and using Theorem 1 in \cite{MZ1}, we see that $V$ is a highest
weight module.  By the condition that $E_{ij}$ acts locally
nilpotently on $V$ for any $i,j\in\Z$ with $1\leq j<i\leq d$ and
using the same Theorem, we can see that $V$ is a lowest weight
module. Thus $V$ is finite dimensional.
\end{proof}


Now we are ready to prove our main result in this paper.

\begin{theorem}Let $V$ be an infinite dimensional irreducible module over $\gl_d$
 on  which the identity matrix acts as a scalar
$b\in\C$. Then $F^\alpha_b(V)$ is an irreducible module over
$\mathcal{W}_d$ for any $\alpha\in\C^d$.
\end{theorem}

\begin{proof} From Lemma 2.3 we know that the action $E_{st}$ on $V$ is
injective for some $s, t: 1\le s\ne t\le d$.

Suppose $W$ is a nonzero  $\mathcal{W}_d$-submodule of
$F^\alpha_b(V)$. Next we will show that $W=F^\alpha_b(V)$. Since
$F^\alpha_b(V)$ is a weight module over $\mathcal{W}_d$, then
$W=\oplus_{n\in\Z^d}M_n\otimes t^n$, where $W_n$ are subspaces of
$V$. Let $W^*=\bigcap_{n\in\Z^d}W_n$. We will show that $W^*$ is a
nonzero $\gl_d$-submodule of $V$.

\

\noindent\textbf{Claim 1.} $W^*$ is nonzero.

Since $W$ is nonzero, without loss of generality, we assume $W_0\neq
0$. Choose a nonzero vector $v\in W_0$, fix some $m\in\Z^d$. For any
$n\in\Z^d, u\in\C^d$, we have
\begin{equation*}\begin{split}
&D(u, m-n)D(u,n)v(0)\\
= & \Big((u \mid  \alpha+n)+
 (m-n)u^T\Big)\Big((u \mid \alpha)+
 nu^T\Big)v(m), \\
 = & \Big((u \mid  \alpha+n)+
 \sum_{i,j}(m_i-n_i)u_jE_{ij}\Big)\Big((u \mid \alpha)+
 \sum_{k,l}n_ku_lE_{kl}\Big)v(m),\\
 = & \Big(\sum_{i,j}\sum_{k,l} (m_i-n_i)u_j n_ku_lE_{ij}E_{kl}+(u \mid \alpha)\sum_{i,j}(m_i-n_i)u_jE_{ij}\\
 & +(u \mid  \alpha+n) \sum_{k,l}n_ku_lE_{kl}+(u \mid \alpha)(u \mid n+\alpha)\Big) v(m)\\
 =& \Big(\sum_{i,j}\sum_{k,l} (m_i-n_i)u_j n_ku_lE_{ij}E_{kl}+\sum_{i,j}((u \mid \alpha)m_i+
 (u \mid  n)n_i)u_jE_{ij}\\
 & +(u \mid \alpha)(u \mid n+\alpha)\Big) v(m).
\end{split}\end{equation*}

 Set \begin{equation*}\begin{split}
 g(n,u)= & \sum_{i,j}\sum_{k,l} (m_i-n_i)u_j n_ku_lE_{ij}E_{kl}\\
 & +\sum_{i,j}((u \mid \alpha)m_i+ (u \mid  n)n_i)u_jE_{ij}\\
 &  +(u \mid \alpha)(u \mid n+\alpha)
  \end{split}\end{equation*} which is a polynomial in $n_i, u_j, 1\leq i,j \leq d$ with coefficients in
  $\text{End}(V)$. Then  $g(n,u)v\in W_m$ for any
  $n\in\Z^d, u\in\C^d$.

Since $s\neq t$, the coefficient of $n_s^2u_t^2$  in $g(n,u)$ is
$-E_{st}E_{st}$.  By Lemma \ref{l},
 we can deduce that $E_{st}E_{st}v\in W_m$ for any $m\in\Z^d$. Thus $E_{st}E_{st}v\in W^*$.
 By the fact that the action $E_{st}$ on $V$ is injective, we see that $W^*\neq 0$. Claim 1 follows.

\

\noindent\textbf{Claim 2.} $W^*$ is a $\gl_d$-submodule of $V$.

By the definition of $W^*$,  we see that  $v(m-n)\in W_{m-n}\otimes
x^{n-m}$ for any $n, m\in \Z^d$, $v\in W^*$. For any $n\in\Z_+^d,
u\in\C^n$ we compute that
\begin{equation*}\begin{split}
&D(u, n)v(m-n)\\
= & \Big((u \mid  \alpha+m-n)+nu^T)v(m), \\
 = &((u \mid  \alpha+m)-\sum_{i}u_in_i+\sum_{i, j}n_iu_jE_{ij})v(m)\in W_m\otimes x^m.
\end{split}\end{equation*}
Considering the coefficient of $n_iu_j$ in the right hand of the
last equality and using Lemma 2.2 again, we know that $E_{ij}v\in
W_m$ for all $m\in\Z^d$ and $i, j : 1\leq i, j\leq d$. This implies
that $E_{ij}v\in W^*$ for all $v\in W^*$ and $i, j : 1\leq i, j\leq
d$. Therefore $W^*$ is a $\gl_d$-submodule of $V$.

By Claim 1 and Claim 2, we deduce that $W^*=V$. Therefore
$W=F^\alpha_b(V)$, hence $F^\alpha_b(V)$ is an irreducible module
over $\mathcal{W}_d$.
\end{proof}



Combining  Theorem 2.4 and Theorem \ref{t}, we obtain the following:

\begin{theorem} Let $\alpha\in \mathbb{C}^d, b\in\C$, and $V$ be
 an irreducible module over $\gl_d$ on  which the identity matrix acts as the scalar
$b$. Then $F^\alpha_b(V)$ is a reducible module over $\mathcal{W}_d$
if and only if $V$ is  isomorphic to the finite dimensional module
whose highest weight is the fundament weight $\omega_k$ and $b=k$,
where $k\in\Z$ with $1\leq k\leq d-1$.
\end{theorem}

Next we will give  a concrete example  of $F^\alpha_b(V)$ for which
$V$ is infinite dimensional.

\noindent{\bf Example.} In \cite{N}, Nilsson constructed a class of irreducible non-weight modules over
$\sl_d$. Explicitly, for $1\leq k\leq d-1$, let $h_k=E_{k,k}-\frac{1}{d}\sum\limits_{i=1}^dE_{ii}$. Then for each
$\beta\in\C$, the polynomial algebra $\C[h_1,\dots,h_{d-1}]$ is an irreducible
$\sl_d$-module under the action
\begin{equation*}\begin{split}h_i\cdot f(h_1,\dots,h_{d-1})=&\ h_i f(h_1,\dots,h_{d-1}),\\
E_{i,d}\cdot f(h_1,\dots,h_{d-1})=&\ (\beta+\sum_{k=1}^{d-1}h_k)(h_i-\beta-1)f(h_1,\dots,h_i-1,\dots,h_{d-1}),\\
E_{d,j}\cdot f(h_1,\dots,h_{d-1})=&\ -f(h_1,\dots,h_j+1,\dots,h_{d-1}),\\
E_{i,j}\cdot f(h_1,\dots,h_{d-1})=&\ (h_i-\beta-1)f(h_1,\dots,h_i-1,\dots,h_j+1,\dots,h_{d-1}),
\end{split}\end{equation*}
where $1\leq i,j\leq d-1$. By Theorem 2.4, $F^\alpha_b(\C[h_1,\dots,h_{d-1}])$ is an irreducible
weight $\mathcal{W}_d$-module with infinite dimensional weight spaces.

\

Finally, we can consider the isomorphism between two modules
described above.
\begin{theorem} Let  $\alpha,\beta\in\C^d$, $b, b'\in\C$,  and $V, W$ be two irreducible $\sl_d$-modules.
Then the $\mathcal{W}_d$-modules $F^\alpha_b(V)$ and
$F^\beta_{b'}(V')$ are isomorphic if and only if $b=b',
\alpha-\beta\in\Z^d$ and $V\cong V'$ as $\sl_d$-modules.
\end{theorem}
\begin{proof}The sufficiency of the condition is clear. Now suppose that $\rho :
F^\alpha_b(V)\rightarrow F^\beta_{b'}(V')$ a $\mathcal{W}_d$-module
isomorphism. Since isomorphic modules have the same weight set, we
have $\alpha+\Z^d=\beta+\Z^d$, i.e., $\alpha-\beta\in\Z^d$. Without
loss of generality, we assume $\alpha=\beta$. Then $\rho$ induces a
linear isomorphism $\sigma$ from $V$ to $V'$ such that
$$\rho(v(0))=\sigma(v)(0), \ \ \forall\  v\in\  V.$$
From $\rho(D(u,n)v(0))=D(u,n)\rho(v(0))$, we see that
\begin{equation}\sum_{i,j}n_iu_j\sigma(E_{ij}v)=\sum_{i,j}n_iu_jE_{ij}\sigma(v)\end{equation}for
all $n\in\Z^d, u\in\C^d$. Comparing the coefficient of $n_iu_j$ on
both sides of the formula (2.2), we know that
$\sigma(E_{ij}v)=E_{ij}\sigma(v) $ for any $v\in V'$ and $1\leq
i,j\leq d$. Therefore
 $V\cong V'$ as $\gl_d$-modules.
\end{proof}

\

\begin{center}
\bf Acknowledgments
\end{center}

\noindent G.L. is partially supported by NSF of China (Grant
11301143) and  the school fund of Henan University(2012YBZR031, 0000A40382); K.Z. is partially supported by  NSF of China (Grant
11271109) and NSERC.


\vspace{1mm}

\noindent   \noindent G.L.: College of Mathematics and Information
Science, Henan University, Kaifeng 475004, China. Email:
liugenqiang@amss.ac.cn

\vspace{0.2cm} \noindent K.Z.: Department of Mathematics, Wilfrid
Laurier University, Waterloo, ON, Canada N2L 3C5,  and College of
Mathematics and Information Science, Hebei Normal (Teachers)
University, Shijiazhuang, Hebei, 050016 P. R. China. Email:
kzhao@wlu.ca


\begin{thebibliography}{99999}

\bibitem[B]{B} Y. Billig, Jet modules, Canad. J. Math., 59 (2007), no.4, 712-729.
\bibitem[BB]{BB} S. Berman and Y. Billig, Irreducible representations for toroidal Lie algebra,
{\it J. Algebra}, 221(1999), 188-231

\bibitem[BF]{BF} Y. Billig and V. Futorny, Classification of simple $W_n$-modules with
finite-dimensional weight spaces,
preprint, arXiv:1304.5458v1.
\bibitem[BMZ]{BMZ} Y. Billig, A. Molev, R. Zhang, Differential equations in vertex algebras and
simple modules for the Lie algebra of vector fields on a torus, Adv.
Math., 218(2008), no.6, 1972-2004.
\bibitem[BZ]{BZ} Y. Billig and K. Zhao, Weight modules over exp-polynomial Lie algebras, J. Pure Appl.
Algebra, 191(2004), 23-42
\bibitem[DMP]{DMP}I. Dimitrov, O. Mathieu, I. Penkov, On the structure of weight modules.
Trans. Amer. Math. Soc. 352 (2000), no. 6, 2857-2869.
\bibitem[E1]{E1} S. Eswara Rao, Irreducible representations of
the Lie-algebra of the diffeomorphisms of a $d$-dimensional torus,
J. Algebra, {\bf 182}  (1996),  no. 2, 401--421.
\bibitem[E2]{E2} S. Eswara Rao, Partial classification of modules for Lie algebra of
diffeomorphisms of d-dimensional torus, J. Math. Phys., 45 (8),
(2004) 3322-3333.

 \bibitem[GLZ]{GLZ} X. Guo, G. Liu and K. Zhao, Irreducible Harish-Chandra modules over extended
Witt algebras, Ark. Mat., 52 (2014), 99-112.
 \bibitem[GZ]{GZ} X.Guo and K.Zhao, Irreducible weight modules over Witt algebras, Proc. Amer. Math. Soc.,
139(2011), 2367-2373.
\bibitem[L1]{L1} T. A. Larsson, Multi dimensional Virasoro algebra, Phys.
Lett.,
B 231, 94-96(1989).
\bibitem[L2]{L2} T. A. Larsson, Central and
non-central extensions of multi-graded Lie algebras, J. Phys., A 25,
1177-1184(1992).
\bibitem[L3]{L3} T. A. Larsson,  Conformal fields:
A class of representations of Vect (N), Int. J. Mod. Phys., A 7,
6493-6508(1992).
\bibitem[L4]{L4} T. A. Larsson, Lowest energy
representations of non-centrally extended diffeomorphism algebras,
Commun. Math. Phys., 201, 461-470(1999).
\bibitem[L5]{L5} T. A.
Larsson, Extended diffeomorphism algebras and trajectories in jet
space, Commun. Math. Phys., 214, 469-491(2000).
\bibitem[M]{M} O. Mathieu; Classification of irreducible weight modules. Ann. Inst.
Fourier (Grenoble) 50 (2000), no. 2, 537-592.



\bibitem[MZ1]{MZ1}  V. Marzuchuk, K. Zhao, Characterization of simple highest weight
modules. Can. Math. Bull. 56(3), 606-614 (2013)
\bibitem[MZ2]{MZ2} V. Marzuchuk and K. Zhao, Supports of weight modules over Witt
algebras, Proc. Roy. Soc. Edinburgh Sect., A 141(2011), no. 1,
155-170.
\bibitem[N]{N} J. Nilsson, Simple $\sl_{n+1}$-module structures on $U(h)$, arXiv:1312.5499.

\bibitem[Sh]{Sh} G. Shen, Graded modules of graded Lie algebras of
Cartan type. I. Mixed products of modules,  Sci. Sinica Ser., A {\bf
29}  (1986),  no. 6, 570-581.
\bibitem[TZ]{TZ}H. Tan, K. Zhao,  $\mathcal{W}_n^+$ and
$\mathcal{W}_n$-module structures on $U(h)$, arXiv:1401.1120.
\bibitem[Z]{Z} K. Zhao, Weight modules over generalized Witt algebras with
1-dimensional weight spaces, Forum Math., Vol.16, No.5, 725-748.

\end{thebibliography}
\end{document}